\documentclass[a4paper,12pt]{amsart}
\usepackage{amsfonts}
\usepackage{amsmath}
\usepackage{amssymb}
\usepackage[a4paper]{geometry}
\usepackage{mathrsfs}
\usepackage{xcolor}
\usepackage{hyperref}
\renewcommand\eqref[1]{(\ref{#1})} 
\usepackage[utf8]{inputenc}
\usepackage[T1]{fontenc}

\numberwithin{equation}{section}
\theoremstyle{plain}
\newtheorem{theorem}{Theorem}[section]

\newtheorem{cor}[theorem]{Corollary}

\theoremstyle{definition}
\newtheorem{definition}[theorem]{Definition}
\newtheorem{rem}[theorem]{Remark}

\begin{document}
\title[Blow-up criteria for the semilinear parabolic equations]
{Blow-up criteria for the semilinear parabolic equations driven
by mixed local-nonlocal operators}

\author[V. Kumar]{Vishvesh Kumar}
\address{
  Vishvesh Kumar:
  \endgraf Department of Mathematics: Analysis, Logic and Discrete Mathematics
  \endgraf Ghent University \endgraf Krijgslaan 281, Building S8,	B 9000 Ghent, Belgium
  \endgraf {\it E-mail address} {\rm  Vishvesh.Kumar@UGent.be/vishveshmishra@gmail.com}
  }
\author[B. T. Torebek ]{Berikbol T. Torebek} \address{Berikbol T. Torebek  \endgraf Institute of
Mathematics and Mathematical Modeling \endgraf 28 Shevchenko str.,
050010 Almaty, Kazakhstan.} \email{torebek@math.kz}



\begin{abstract} The main goal of this paper is to establish \emph{necessary and sufficient conditions} for the nonexistence of a global solution to the semilinear heat equation with a mixed local--nonlocal operator $ -\Delta + (-\Delta)^\sigma$, under a general time-dependent nonlinearity. Our results complement the recent work of Carhuas-Torre et al. [\textit{ArXiv, (2025), arXiv:2505.20401}],
in which the authors provide sufficient conditions for the existence and nonexistence of global solutions. In particular, our results recover the critical Fujita exponent for time-independent power-type nonlinearities, as obtained by Biagi et al. [\textit{Bull. London Math. Soc.} (2024), 1--20] and Del Pezzo et al. [\textit{Nonlinear Anal.} 255 (2025), 113761].

\end{abstract}

\maketitle

\section{Introduction and main result}
This paper investigates the blow-up dynamics of solutions to the semilinear mixed local–nonlocal heat equation:
\begin{equation}\label{mainproblem}
\begin{cases}u_{t}(x, t)+\mathcal{L}u(x, t)=h(t) f(u(x, t)), & (x, t) \in \mathbb{R}^{d} \times(0, \infty),  \\{}\\ u(x, 0)=u_{0}(x) \geq 0, & x \in \mathbb{R}^{d},\end{cases}
\end{equation}
where $u_{0}\geq 0,\,u_0\not\equiv 0,$ $f:\mathbb{R}_+\rightarrow\mathbb{R}_+$ is a locally Lipschitz continuous function such that $f(0)=0$ and $f(u)>0$ for all $u>0$, $0\leq h\in C([0,\infty)]),$ and $$\mathcal L = -\Delta+(-\Delta)^{\sigma},$$ and $(-\Delta)^{\sigma}$ stand for the (nonlocal) fractional Laplacian of order $\sigma\in (0,1)$:
\begin{align*}
(-\Delta)^{\sigma} v(x) & = C_{d,\sigma}\cdot \mathrm{P.V.}\int_{\mathbb{R}^d}\frac{v(x)-v(y)}{|x-y|^{d+2\sigma}}\,dy
\\
& = C_{d,\sigma}\cdot\lim_{\varepsilon\to 0^+}\int_{\{|x-y|\geq\varepsilon\}}\frac{u(x)-u(y)}{|x-y|^{d+2\sigma}}\,dy,\,C_{d,\sigma}=\frac{2^{2\sigma-1}{2\sigma}\Gamma\left(\frac{d+2\sigma}{2}\right)}{\pi^{d/2}\Gamma(1-\sigma)}.
\end{align*}
The elliptic and parabolic problems related with the local-nonlocal operator $\mathcal{L}$ have been studied intensively in recent years in several different contexts by many
prominent researchers. This interest stems from the operator's diverse applications across several areas of mathematics, notably in probability theory and mathematical biology. In the context of probability theory, the mixed operator naturally arises as a combination of a classical (Brownian) stochastic process and a long-range (L\'evy) stochastic process. Meanwhile, in mathematical biology, it serves as a tool for modeling optimal foraging strategies in animals. For a more comprehensive discussion, we refer the reader to \cite{Chen1, Chen2, Dip1, Dip2, Dip3, Dip4} and the references therein.

In the classical paper  \cite{Fujita}, Fujita considered the problem \eqref{mainproblem} for $$f(u)=u^p,\,\, h=1\,\,\, \text{and} \,\,\,\mathcal{L}=-\Delta,$$ to discuss the conditions for the existence and
nonexistence of the global solutions and proved that: for any $u_0\geq 0$ the problem \eqref{mainproblem} possesses no global positive solutions if $$1<p<p_F=1+\frac{2}{d},$$ while there exists a positive global  solution of \eqref{mainproblem} if $p>p_F$ and $u_0\geq 0$ is smaller than a Gaussian. Later, Hayakawa \cite{Hayakawa} and Sugitani \cite{Sugitani} established that $p=p_F$ also belongs to the blow-up case for nonnegative initial data. The number $p_F=1+\frac{2}{d}$ is called the {\it Fujita (critical) exponent}. 

On the other hand, many authors studied the problem \eqref{mainproblem} for $f(u)=u^p,\, h=1$ and for the nonlocal fractional Laplacian $\mathcal{L}=(-\Delta)^\sigma,\,\sigma\in (0,1);$ we refer to \cite{Fino, Guedda, Naga, Sugitani} and reference therein for more details. It was observed in these studies that that the critical exponent for the problem \eqref{mainproblem} for $f(u)=u^p,\, h=1$ and $\mathcal{L}=(-\Delta)^\sigma$ in the sense of Fujita is $p_F=1+\frac{2 \sigma}{d}.$ 

In his pioneering work \cite{Meier}, Meier considered \eqref{mainproblem} for $\mathcal{L}=-\Delta,$ $f(u)=u^p,\, p>1$ and showed that if there exists $0 \leq u_0 \in L^\infty(\mathbb{R}^n)$ such that 
\begin{equation} \label{cond1.2}
    \int_0^\infty h(t) \|e^{t \Delta} u_0\|_{L^\infty (\mathbb{R}^d)}^{p-1} dt <\infty,
\end{equation}
then there exists a global solution to \eqref{mainproblem} with $\lim_{t \rightarrow \infty} \|u(t)\|_{L^\infty (\mathbb{R}^d)} =0$. On the other hands, he also showed that if 
\begin{equation}
    \limsup_{t \rightarrow \infty } \|e^{t \Delta} u_0\|_{L^\infty (\mathbb{R}^d)}^{p-1} \int_0^t h(t)\, dt =\infty ,
\end{equation}  then there is no global solution to problem \eqref{mainproblem} for  any nontrivial nonnegative initial data $u_0 \in L^\infty(\mathbb{R}).$ One of the distinguished features of Meier's work is the criterion based solely on 
the computation of \( \|e^{t \Delta} u_0\|_{L^\infty (\mathbb{R}^d)} \) to determine whether the solution is global or blows up in finite time. The condition \eqref{cond1.2} was in fact used by Weissler in the case \( h = 1 \)  to obtain a nonnegative global solution of \eqref{mainproblem}. These results were further extended  by Loayza and Paixão~\cite{LP} to the case where \( f \) is a locally Lipschitz 
function with \( f(0)=0 \), under the additional assumption that both \( f \) and  \( f(s)/s \) are nondecreasing on \( (0,a] \) for some \( a>0 \). It was further extended for a system of two equations in \cite{Castillo1,Castillo2}. Based on these  results, two sufficient conditions for the existence and nonexistence of global 
solutions, respectively, are obtained. However, a necessary and sufficient condition  for the existence of global solutions to \eqref{mainproblem} remained an open  question until recently, when Laister and Sier\.z\k{e}ga  \cite{Laister}    established a necessary 
and sufficient condition for the nonexistence of global solutions to \eqref{mainproblem}  in the case \( \mathcal{L} = (-\Delta)^\sigma \) for $\sigma \in (0, 1]$ and $h=1.$ This was further extended by Chung and Hwang~\cite{Chung} to provided necessary and sufficient conditions for nonexistence of global solution  \eqref{mainproblem} in the general time-dependent nonlinearity.  In \cite{CKR}, the authors extended these results in the setting of unimodular Lie groups with $\mathcal{L}$ being the sum of square of H\"ormander vector fields providing the necessary and sufficient condition for the three-dimensional Heisenberg
group.

Recently, Biagi, Punzo, and Vecchi \cite{Biagi} investigated the problem \eqref{mainproblem} with \( h(t) = 1 \) and \( f(u) = |u|^p \), and identified the Fujita-type critical exponent \( 1 + \frac{2\sigma}{d} \) that characterizes the threshold between existence and nonexistence of positive global solutions for nonnegative initial data \( u_0 \geq 0 \).  Similar results were obtained by Del Pezzo and Ferreira in \cite{Pezzo} for the operator $\mathcal L_{a,b} = -a\Delta+b(-\Delta)^s,$  where $a,b\in\mathbb{R}_+.$ In \cite{KT2025}, we extended the analysis in \cite{Biagi, Pezzo} by considering the same setting, i.e., \( h(t) = 1 \) and \( f(u) = |u|^p \), but with the addition of an external forcing term as well as studying the sign-changing solutions. Furthermore, in \cite{Kirane}, the authors analyzed the asymptotic behavior of solutions to \eqref{mainproblem} in the case where the nonlinearity takes the form \( h(t)|u|^p \). Very recently, the authors \cite{Castillo3} have presented sufficient conditions concerning the global and nonglobal existence of solutions to \eqref{mainproblem}, extending the work of Loayza and Paixão~\cite{LP} for mixed local and nonlocal operators. We also refer \cite{BM25} for further developments in this line of research.

The primary objective of this paper is to establish blow-up criteria for solutions to problem \eqref{mainproblem}, that is, to establish \emph{necessary and sufficient conditions} for the nonexistence of global solutions to problem \eqref{mainproblem}.

In preparation for subsequent developments, we first present the minorant and majorant functions. 
\begin{definition} For a given function $f$ we define the minorant $f_{m}:[0, \infty) \rightarrow[0, \infty)$ and the majorant $f_{M}:[0, \infty) \rightarrow[0, \infty)$ as follows
$$
\begin{aligned}
f_{m}(u) & :=\inf_{0<\alpha<1} \frac{f(\alpha u)}{f(\alpha)}, & & u \geq 0, \\
f_{M}(u) & :=\sup_{0<\alpha<1} \frac{f(\alpha u)}{f(\alpha)}, & & u \geq 0 .
\end{aligned}
$$
\end{definition}
This definition leads us to the inequality given below \begin{equation}\label{f} f(\alpha) f_{m}(u) \leq f(\alpha u) \leq f(\alpha) f_{M}(u),\,\, 0<\alpha<1,\,\, u\geq 0.\end{equation}

Abstract formulations of the minorant and majorant functions were introduced in \cite{Laister}, with further properties examined in \cite{Chung1}. We will also need some extra conditions for minorant and majorant functions:
\begin{equation}\label{1.5}
\int_{1}^{\infty} \frac{d u}{f_{m}(u)}<\infty \text { and } \lim _{u \rightarrow 0+} \frac{f_{M}(u)}{u}=0.
\end{equation}
These conditions are natural, since in order to address the existence or 
nonexistence of global solutions, certain growth assumptions on the function 
$ f $ near the origin are necessary; we refer to \cite{BB98} and reference therein for more details.

It is well known that if $f$ is a nonnegative locally Lipschitz continuous function, then problem \eqref{mainproblem} (for the abstract case of the operator $\mathcal{L}$) possesses a nonnegative local-in-time mild solution $u\in C\left([0,T_{\max});L^\infty(\mathbb{R}^d)\right),\,\,T_{\max}\leq \infty$ for every nonnegative nontrivial initial data $u_0 \in L^\infty(\mathbb{R}^d)$ satisfying
\begin{equation}\label{mild}
u(x, t)=e^{-t\mathcal{L}}u_{0}(x)+\int_{0}^{t} e^{-(t-s)\mathcal{L}} h(s) f(u(x, s)) d s
\end{equation}
for all $t \in [0, T_{ \max});$ (see \cite{pazy, QS}).  When $T_{\max}=\infty,$ such a solution is called global mild solution. Furthermore, if $T_{\max}<\infty,$, then the local mild solution $u$ blows up in finite time, i.e. $$\lim\limits_{t\rightarrow T_{\max}}\|u(\cdot, t)\|_{L^\infty(\mathbb{R}^d)}=\infty.$$

We state the main result of this paper below.
\begin{theorem}\label{th1} Let \(h \in C([0,\infty))\) with \(h \geq 0\). Assume that 
\(f \in C([0,\infty))\) is a nonnegative convex function with 
\(f(0)=0\) and \(f(u)>0\) for all \(u>0\) such that \eqref{1.5} holds. 
Then the following statements are equivalent:
\begin{itemize}
\item[(i)] for every  $0< u_{0} \in C_0\left(\mathbb{R}^{d}\right)\cap L^1\left(\mathbb{R}^{d}\right)$, we have $$\int_{0}^{\infty} h(t) \frac{f\left(\left\|e^{-t\mathcal{L}} u_{0}\right\|_{L^\infty (\mathbb{R}^d)}\right)}{\left\|e^{-t\mathcal{L}}u_{0}\right\|_{L^\infty (\mathbb{R}^d)}} d t=\infty;$$
\item[(ii)] for any real number $\epsilon>0,$ we have \begin{equation}\label{(ii)}\int_{1}^{\infty} h(t) t^{\frac{d}{2\sigma}} f\left(\epsilon t^{-\frac{d}{2\sigma}}\right) d t=\infty;\end{equation}
\item[(iii)] the mild local solutions of problem \eqref{mainproblem} blows-up in finite time for any $0< u_{0} \in C_0\left(\mathbb{R}^{d}\right)\cap L^1\left(\mathbb{R}^{d}\right)$.
\end{itemize}
Here $e^{-t\mathcal{L}}$ is the heat semigroup on $L^{q}\left(\mathbb{R}^{d}\right)(q \geq 1)$ generated by $\mathcal{L}$.
\end{theorem}

\begin{rem}
Theorem \ref{th1} provides a necessary and sufficient condition for the finite-time blow-up of the solution to problem \eqref{mainproblem}. This implies that, rather than directly establishing the blow-up of the solution, it suffices to verify the more readily computable condition (ii). Moreover, Theorem \ref{th1} extends the contribution of Carhuas-Torre et al. by complementing their recent result in \cite{Castillo3}, in which sufficient conditions for the blow-up of solutions to problem \eqref{mainproblem} were established. It is worth noting that Theorem \ref{th1} extends the recent results of Chung et al. \cite{Chung} for the mixed local–nonlocal operator.
\end{rem}

As an application of Theorem \ref{th1} we deduce the following interesting results recovering the Fujita exponent discussed in the main result of \cite{Biagi}.
\begin{cor}\label{cor1} Suppose that $h=1$ and $f=u^p (p>1).$ Then, problem \eqref{mainproblem} has no global solution for any positive initial data $u_0 \in C_0\left(\mathbb{R}^{d}\right) \cap L^1(\mathbb{R}^d)$ if and only $$1 < p \leq  p_F:=1+\frac{2\sigma}{d}.$$ 
\end{cor}

Another interesting consequence of Theorem~\ref{th1} is the following result, 
which complements \cite[Corollary 7]{Castillo3}, since the behaviour of the solution 
at the Fujita critical exponent was not addressed there.

\begin{cor}\label{cor2} Assume that $$h(t)=t^r (r \geq 0)\,\,\,\, \text{and} \,\,\,\,f=(1+u)[\ln(1+u)]^p (p >1).$$ Then, problem \eqref{mainproblem} has no global solution for any positive initial data $u_0 \in C_0\left(\mathbb{R}^{d}\right) \cap L^1(\mathbb{R}^d)$ if and only $$1 < p \leq  p_F:=1+\frac{2 \sigma(1+r)}{d}.$$ 
\end{cor}

The following example was considered in \cite{Chung} in case $\mathcal{L}=(-\Delta)^\sigma$, which demonstrates the critical exponent for the combined non-linearity $u^p+u^q$.
\begin{cor}\label{cor3} Suppose that $$h(t)=t^r+t^s (r \geq s \geq 0)\,\,\,\, \text{and} \,\,\,\,f=u^p+u^q (p \geq q >1).$$ Then, problem \eqref{mainproblem} has no global solution for any positive initial data $u_0 \in C_0\left(\mathbb{R}^{d}\right) \cap L^1(\mathbb{R}^d)$ if and only $$1 < p \leq  p_F:=1+\frac{2 \sigma(1+s)}{d}.$$ 
\end{cor}

Finally, we present an example that illustrates a particularly interesting phenomenon: the absence of a Fujita-type critical exponent.
\begin{cor}\label{cor4} Assume that $h(t)=e^{\theta t}$ $(\theta \neq 0)$ and $f=u^p (p >1).$ Then, the mild solution of problem \eqref{mainproblem} blows-up in finite-time for any positive initial data $u_0 \in C_0\left(\mathbb{R}^{d}\right) \cap L^1(\mathbb{R}^d)$ if and only $\theta>0.$ 
\end{cor}
\begin{rem}
From Corollary \ref{cor4}, the following conclusion can be drawn: for any $p>1$ a global solution exists whenever $\theta<0$, whereas for any $p>1$ the solution blows up in finite time if $\theta>0$. This indicates that the decisive parameter governing the qualitative behavior of solutions is not the exponent $p$ as in the case $\theta=0$, but rather the parameter $\theta$.
\end{rem}

Apart from the introduction, this note is organized as follows. In the next section, 
we present the proof of Theorem~\ref{th1}, while the final section is devoted to the proof Corollaries \ref{cor1}-\ref{cor4}.

\section{Proof of Main results} \label{sec2}
Let us denote by $e^{-t\mathcal{L}}$ the heat semigroup on $L^{q}\left(\mathbb{R}^{d}\right)(q \geq 1)$ generated by $-\mathcal{L}$ such that
$$
e^{-t\mathcal{L}} \phi(x)=\int_{\mathbb{R}^{d}} p_{t}(x-y) \phi(y) d y, \quad \phi \in L^{q}\left(\mathbb{R}^{d}\right),
$$
where \begin{equation}p_{t}(x)=\frac{1}{(4\pi t)^{d/2}}\int_{\mathbb{R}^d}e^{-\frac{|x-\xi|^2}{4t}}H_t^{\sigma}(\xi)d\xi,\,t>0,\,x\in\mathbb{R}^d\end{equation} is the heat kernel of $\mathcal{L}$. Here $H_t^s$ is the heat kernel of the fractional Laplacian $(-\Delta)^\sigma.$
It is known that the following properties hold for $p_t(x)$ (see \cite{Biagi, Castillo3, Kirane, BM25}):
\begin{itemize}
    \item[({\bf p1})] $p_t$ is a positive, $p_t\in C^\infty((0,\infty)\times\mathbb{R}^d),$ $p_t(x)=p_t(-x),\,t>0,\,x\in\mathbb{R}^d$ and $$\int_{\mathbb{R}^d}p_t(y)dy=1,\,t>0.$$
    \item[({\bf p2})] It holds that $$\int_{\mathbb{R}^d}p_t(x-y)p_s(y)dy=p_{t+s}(x)$$ for every fixed $x\in\mathbb{R}^d$ and $t,s>0.$ In addition,
   \begin{equation}\label{C}
    0<p_t(x)\leq C^*t^{-\frac{d}{2\sigma}},\
    C^*>0,\,\,\,\,\,t>0,\,x\in\mathbb{R}^d.
   \end{equation}
   \item[({\bf p3})] There exists $C_*>0$ such that
   \begin{equation}\label{(p3)}
    p_t(x)\geq C_*t^{-\frac{d}{2\sigma}},\,\,t>1,\,|x|\leq \sqrt{t}.
   \end{equation}
\end{itemize}

We now begin the proof of our main result, Theorem \ref{th1}. To this end, let us first recall a basic fact that will be used occasionally throughout the argument. In the statement of Theorem \ref{th1}, we assume that \(f \in C([0,\infty))\) is a nonnegative convex function satisfying \(f(0)=0\) and \(f(u)>0\) for all \(u>0\). Under these assumptions, it follows that the mapping  $u \mapsto \frac{f(u)}{u}$ is non-decreasing. Consequently, the function  $f(u) = \frac{f(u)}{u}\, u $ is also non-decreasing (see \cite[Remark 4.7]{CKR}).

\begin{proof}[Proof of Theorem \ref{th1}] (i) $\Rightarrow$ (ii) : Let $\epsilon>0$ be the arbitrary small real number. Assume that $u_{0} \in C_0\left(\mathbb{R}^{d}\right) \cap L^{1}\left(\mathbb{R}^{d}\right)$ is the positive function with $$C^*\left\|u_{0}\right\|_{L^{1}\left(\mathbb{R}^{d}\right)}\leq\epsilon,$$ where $C^*>0$ is the constant in \eqref{C}. Thus, the inequality  \eqref{C} implies that
$$
\begin{aligned}
\left\|e^{-t\mathcal{L}}u_{0}\right\|_{L^\infty(\mathbb{R}^d)} & =\sup _{x \in \mathbb{R}^{d}} \int_{\mathbb{R}^{d}} p_t(x-y) u_{0}(y) d y \\
& \leq C^* t^{-\frac{d}{2\sigma}}  \int_{\mathbb{R}^{d}} u_{0}(y) d y \\
& \leq \epsilon t^{-\frac{d}{2\sigma}}
\end{aligned}
$$
for every $t \geq 1$. The convexity of $f$ implies that $\frac{f(u)}{u}$ is nondecreasing in $u$. Therefore, we have 
$$
\frac{f\left(\left\|e^{-t\mathcal{L}} u_{0}\right\|_{L^\infty(\mathbb{R}^d)}\right)}{\left\|e^{-t\mathcal{L}} u_{0}\right\|_{L^\infty(\mathbb{R}^d)}} \leq \frac{f\left(C^*\left\|u_{0}\right\|_{L^{1}\left(\mathbb{R}^{d}\right)} t^{-\frac{d}{2\sigma}}\right)}{C^*\left\|u_{0}\right\|_{L^{1}\left(\mathbb{R}^{d}\right)} t^{-\frac{d}{2\sigma}}} \leq \frac{f\left(\epsilon t^{-\frac{d}{2\sigma}}\right)}{\epsilon t^{-\frac{d}{2\sigma}}}$$
for every $t\geq 1$. Thus, we conclude, with our assumption that (i) holds, that
\begin{equation}\label{BT1}\begin{split}\infty&=\int_{0}^{\infty} h(t) \frac{f\left(\left\|e^{-t\mathcal{L}} u_{0}\right\|_{L^\infty(\mathbb{R}^d)}\right)}{\left\|e^{-t\mathcal{L}} u_{0}\right\|_{L^\infty(\mathbb{R}^d)}} d t \\&\leq \int_{0}^{1} h(t) \frac{f\left(\left\|e^{-t\mathcal{L}} u_{0}\right\|_{L^\infty(\mathbb{R}^d)}\right)}{\left\|e^{-t\mathcal{L}} u_{0}\right\|_{L^\infty(\mathbb{R}^d)}} d t+\int_{1}^{\infty} h(t) \frac{f\left(\epsilon t^{-\frac{d}{2\sigma}}\right)}{\epsilon t^{-\frac{d}{2\sigma}}} d t.\end{split}\end{equation}
From ({\bf p1}) it follows that $\left\|e^{-t\mathcal{L}} u_{0}\right\|_{L^\infty(\mathbb{R}^d)}\leq  \|u_{0}\|_{L^\infty(\mathbb{R}^d)},$ and therefore we note that
$$\int_{0}^{1} h(t) \frac{f\left(\left\|e^{-t\mathcal{L}} u_{0}\right\|_{L^\infty(\mathbb{R}^d)}\right)}{\left\|e^{-t\mathcal{L}} u_{0}\right\|_{L^\infty(\mathbb{R}^d)}} d t\leq \frac{f\left(\|u_0\|_{L^\infty(\mathbb{R}^d)}\right)}{\|u_0\|_{L^\infty(\mathbb{R}^d)}}\int_{0}^{1} h(t)dt<\infty.$$
Finally, from \eqref{BT1} we conclude that
$$\int_{1}^{\infty} h(t) \frac{f\left(\epsilon t^{-\frac{d}{2\sigma}}\right)}{\epsilon t^{-\frac{d}{2\sigma}}} d t=\infty,$$
completing the proof.\\
(ii) $\Rightarrow$ (iii): The proof proceeds by contradiction. Assume that $u$ is a global mild solution of problem \eqref{mainproblem} for some initial datum $0<u_0 \in C_0(\mathbb{R}^d) \cap L^1(\mathbb{R}^d).$ We can assume without loss of generality, as $u_0 \in C_0(\mathbb{R}^d),$ that \begin{equation}\label{3.4}
u_{0}(x) \geq {\varepsilon_{1}}C_*^{-1}2^{-\frac{d}{2\sigma}}>0,\,|x|<1\end{equation} for some $\varepsilon_1>0$. Here $C_*>0$ is a constant in \eqref{(p3)}. Invoking \eqref{mild} and employing the semigroup property, we deduce that
\begin{equation}\label{3.2}
e^{-t\mathcal{L}} u(x, t)=U(x, t)+V(x, t),
\end{equation}
with
$$
U(x, t)=e^{-2t\mathcal{L}} u_{0}(x)
$$
and
$$
V(x, t)=\int_{0}^{t} e^{-(2t-s)\mathcal{L}} h(s) f(u(x, s)) d s.
$$
As a first step, we fix $x=x_0$ and show that
\begin{equation}\label{3.3}
U(x_0, t) \geq \varepsilon_{1} t^{-\frac{d}{2\sigma}}, t>1.
\end{equation} Using the \eqref{(p3)} and \eqref{3.4} we have, for $t>1,$ that
\begin{equation*}\begin{split}
U(x_0, t)& =\left.e^{-2t\mathcal{L}} u_{0}(x)\right|_{x=x_0}\\& =
\int_{\mathbb{R}^{d}} p_{2t}(x_0-y) u_{0}(y) d y \\& \geq
\int_{|x_0-y|< \sqrt{2t}} p_{2t}(x_0-y) u_{0}(y) d y \\
& \geq C_*(2 t)^{-\frac{d}{2\sigma}} \int_{|x_0-y|<1} u_{0}(y) d y\\& = C_*(2 t)^{-\frac{d}{2\sigma}} \int_{|y|<1} u_{0}(y) d y \\
& \geq \varepsilon_{1} t^{-\frac{d}{2\sigma}}.\end{split}
\end{equation*}
It follows immediately that, for all $0 \leq s \leq t$, one has $s \leq 2 t-s,$ and thus, properties \eqref{C} and \eqref{(p3)} imply that, for some $C'>0,$ we have
 $$\frac{p_{2t-s}(x)}{p_{s}(x)} \geq C'\frac{(2 t-s)^{-\frac{d}{2\sigma}}}{s^{-\frac{d}{2\sigma}}},\,\,s>1,\,|x|<1,$$ thanks to $|x|<1<\sqrt{s}\leq \sqrt{2t-s}.$
From the above estimate, it follows that
\begin{equation*}\begin{split}
V(x, t)&=\int_{0}^{t} e^{-(2t-s)\mathcal{L}}e^{s\mathcal{L}} e^{-s\mathcal{L}}h(s) f(u(x, s)) d s \\&\geq \int_{1}^{t} e^{-(2t-s)\mathcal{L}}e^{s\mathcal{L}} e^{-s\mathcal{L}}h(s) f(u(x, s)) d s \\&\geq C' \int_{1}^{t} h(s)\left(\frac{s}{2 t-s}\right)^{\frac{d}{2\sigma}} e^{-s\mathcal{L}} f(u(x, s)) d s. \end{split}
\end{equation*}
Subsequently, applying Jensen’s inequality yields that
\begin{equation}\label{3.5}\begin{split}
V(x, t) & \geq C' \int_{1}^{t} h(s)\left(\frac{s}{2 t-s}\right)^{\frac{d}{2\sigma}} f\left(e^{-s\mathcal{L}} u(x, s)\right) \,d s \\
& \geq(2 t)^{-\frac{d}{2\sigma}}  C' \int_{1}^{t} h(s) s^{\frac{d}{2\sigma}} f\left(e^{-s\mathcal{L}} u(x, s)\right)\, d s.\end{split}
\end{equation}
Denote $$W(x, t):=\varepsilon_{1}+2^{\frac{d}{2\sigma}} C' \int_{1}^{t} h(s) s^{\frac{d}{2\sigma}} f\left(e^{-s\mathcal{L}} u(x, s)\right) d s.$$ Then, combining inequalities \eqref{3.2}, \eqref{3.3}, and \eqref{3.5}, we obtain
\begin{equation}\label{(2)}
e^{-t\mathcal{L}} u(x, t) \geq t^{-\frac{d}{2\sigma}} W(x, t),\end{equation}
for all $|x|<1$ and $t>1.$
The global existence of $u$ implies the global existence of $W$. Using the monotonicity of $f$, we deduce using \eqref{(2)} that, for each fixed $x$ and $t>1,$ the following holds
\begin{equation}\label{3.6}\begin{split}
\frac{\partial}{\partial t}W(x, t)&=2^{-\frac{d}{2\sigma}} h(t) t^{\frac{d}{2\sigma}} f\left(e^{-t\mathcal{L}} u(x, t)\right)\\& \geq 2^{-\frac{d}{2\sigma}} h(t) t^{\frac{d}{2\sigma}} f\left(W(x,t) t^{-\frac{d}{2\sigma}}\right).\end{split}
\end{equation} Let $0<\epsilon<1$ be small enough so that $\epsilon<\varepsilon_{1}$. 
By the definition of
the minorant function $f_{m}$ and \eqref{f} one obtain 
\begin{align*}
f\left(W(x,t) t^{-\frac{d}{2\sigma}}\right)&=f\left(\frac{\epsilon W(x,t)} {\epsilon t^{\frac{d}{2\sigma}}}\right)f\left(\epsilon t^{-\frac{d}{2\sigma}}\right)\left(f\left(\epsilon t^{-\frac{d}{2\sigma}}\right)\right)^{-1}\\& \geq f_{m}\left(\frac{W(x, t)}{\epsilon}\right) f\left(\epsilon t^{-\frac{d}{2\sigma}}\right),\end{align*} for $0<\epsilon<1$ and $t>1.$
Also, we note that $f_m(W) = 0$ at the origin and $f(W) \geq 1$ for $W\geq 1$. Since $\epsilon<\varepsilon_1<W$, we have
$$\frac{f\left(\frac{\epsilon W(x,t)} {\epsilon t^{\frac{d}{2\sigma}}}\right)}{f\left(\epsilon t^{-\frac{d}{2\sigma}}\right)}\geq \frac{f\left(\frac{\epsilon \varepsilon_1} {\epsilon t^{\frac{d}{2\sigma}}}\right)}{f\left(\epsilon t^{-\frac{d}{2\sigma}}\right)}\geq 1,$$ which implies $f_{m}\left(\frac{W(x, t)}{\epsilon}\right)\geq 1$ for all $t>1$ and $|x|<1$.

Subsequently, taking into account all the above three inequalities, we come to a conclusion that $$
\frac{\frac{\partial}{\partial t}W(x, t)}{f_{m}\left(\frac{W(x, t)}{\epsilon}\right)} \geq 2^{-\frac{d}{2\sigma}} h(t) t^{\frac{d}{2\sigma}} f\left(\epsilon t^{-\frac{d}{2\sigma}}\right)
$$
holds for fixed $x$ such that $|x|<1$ and $t>1$. Further, integrating the last inequality over $(1,t)$, we have that
\begin{equation}\label{(1)}
\Phi(W(x, 1))-\Phi(W(x, t)) \geq 2^{-\frac{d}{2\sigma}} \epsilon \int_{1}^{t} h(s) s^{\frac{d}{2\sigma}} f\left(\epsilon s^{-\frac{d}{2\sigma}}\right) d s,\end{equation}
where $$\Phi(g):=\int_{\frac{g}{\epsilon}}^{\infty} \frac{d y}{f_{m}(y)},\,\, g>0.$$ It follows that $\Phi$ is a bijective, decreasing, and positive function, hence $\Phi(g)\rightarrow 0$ as ${g \rightarrow +\infty}$. Let us denote $$\gamma(t)=2^{-\frac{d}{2\sigma}} \epsilon \int_{1}^{t} h(s) s^{\frac{d}{2\sigma}} f\left(\epsilon s^{-\frac{d}{2\sigma}}\right) d s,$$ then from \eqref{(ii)} it follows that $\gamma(t)\rightarrow+\infty$ at $t\rightarrow+\infty.$ Since $\Phi$ is positive, one can rewrite the inequality \eqref{(1)} as $$ +\infty= \lim_{t \rightarrow \infty} \gamma(t) \geq  \Phi(W(x,1))\geq \gamma(t), \,t>1.$$
It is easy to check that $\gamma$ is a continuous in $t \in (1,\infty)$, then there exists $\tau\in [t, +\infty)$ such that $\Phi(W(x,1))= \gamma(\tau).$ Then \eqref{(1)} gives us
$$\gamma(\tau)-\Phi(W(x,t))\geq  \gamma(t)$$ for any $t>1.$ Thus, passing to the limit as $t\rightarrow\tau$ in the above inequality, we have $\lim_{t \rightarrow \tau }\Phi(W(x,\tau)) \leq 0,$ which in turn shows that $$\lim_{t \rightarrow \tau }\Phi(W(x,\tau))=0.$$  Using the other properties of $\Phi$ along with $\lim\limits_{g\rightarrow\infty}\Phi(g)=0$, it follows that $\lim\limits_{t\rightarrow\tau}W(x,t)\rightarrow+\infty,$ for some finite time $\tau.$ Hence, by \eqref{(2)}, the solution $u$ cannot be global. This conclude the proof. 
\\
$\mathbf{( i i i )} \Rightarrow \mathbf{( i )}$ : The proof of this part will be given by contradiction and it is based on the monotone sequence argument (see \cite{Castillo1, Castillo2}). Assume that there exists $0< \varphi \in C_0(\mathbb{R}^d) \cap L^1(\mathbb{R}^d),\,\varphi\not\equiv 0,$ such that
$$
K:=\int_{0}^{\infty} h(t) \frac{f\left(\left\|e^{-t\mathcal{L}} \varphi\right\|_{L^\infty(\mathbb{R}^d)}\right)}{\left\|e^{-t\mathcal{L}} \varphi\right\|_{L^\infty(\mathbb{R}^d)}} d t<\infty.
$$ 


To provide a contradiction, we will establish the existence of a global solution $u$ corresponding to certain nonnegative initial data $u_0.$ We define $u_0:=\mu \varphi$, for some $0<\mu<1$ satisfies
\begin{equation} \label{secmu}
    \sqrt{\mu} \|e^{-t\mathcal{L}} \phi\|_{L^\infty(\mathbb{R}^d)}<1.
\end{equation}
It follows from the fact that the majorant function $f_{M}$ satisfies $\lim\limits_{y \rightarrow 0+} \frac{f_{M}(y)}{y}=0,$ that for every $\epsilon>0$ there exists $\delta>0$ such that
\begin{equation}
\sqrt{\mu}(1+K) \leq \delta, \text { then } \frac{f_{M}(\sqrt{\mu}(1+K))}{\sqrt{\mu} (1+K)} <\epsilon. 
\end{equation}
This further implies that 
\begin{equation}\label{3.7}
    0 \leq \frac{f_{M}(\sqrt{\mu}(1+K))}{\sqrt{\mu} } <\epsilon (1+K).
\end{equation}
 Moreover, this also shows that $0<f_{M}(\sqrt{\mu}(1+K))<\infty$. Assume that $\left\{y_{j}\right\}_{(j \geq 0)}$ denote a sequence of positive functions defined by
$$
\left\{\begin{array}{l}
y_{0}(x, t)=e^{-t\mathcal{L}} u_{0}(x), \\{}\\
y_{j}(x, t)=e^{-t\mathcal{L}} u_{0}(x)+\int_{0}^{t} h(s) e^{-(t-s)\mathcal{L}} f\left(y_{j-1}(x, s)\right) d s,\,\, j \geq 1.
\end{array}\right.
$$
Using the principle of mathematical induction, we will prove that
\begin{equation}\label{3.8}
0 \leq y_{j}(x, t) \leq(1+K) e^{-t\mathcal{L}} u_{0},\, (x,t)\in\mathbb{R}^d\times[0,\infty)
\end{equation}
for all $j\geq 0.$
Inequality \eqref{3.8} follows immediately in the case $j=0$, since in this instance we have
$$y_{0}(x, t)=e^{-t\mathcal{L}} u_{0}(x) \leq(1+K) e^{-t\mathcal{L}} u_{0}(x),\,\,x \in \mathbb{R}^{d},\,t\geq 0.$$
Let us assume that inequality \eqref{3.8} remains valid for $j=k$, then
$$
\begin{aligned}
y_{k+1}(x, t) & =e^{-t\mathcal{L}} u_{0}(x)+\int_{0}^{t} h(s) e^{-(t-s)\mathcal{L}} f\left(y_{k}(x, s)\right) d s \\
& \leq e^{-t\mathcal{L}} u_{0}(x)+\int_{0}^{t} h(s) e^{-(t-s)\mathcal{L}} f\left((1+K) e^{-s\mathcal{L}} u_{0}(x)\right) d s.
\end{aligned}
$$
By exploiting the fact that $\frac{f(y)}{y}$ is nondecreasing function  and noting that $e^{-s\mathcal{L}} u_{0}(x)> 0,\,u_0\geq 0$, we arrive at the estimate
$$
\begin{aligned}
y_{k+1}(x, t) & \leq e^{-t\mathcal{L}} u_{0}(x)+\int_{0}^{t} h(s) e^{-(t-s)\mathcal{L}} \frac{f\left((1+K) e^{-s\mathcal{L}} u_{0}(x)\right)}{e^{-s\mathcal{L}} u_{0}(x)} e^{-s\mathcal{L}} u_{0}(x) d s \\
& \leq e^{-t\mathcal{L}} u_{0}(x)+e^{-t\mathcal{L}} u_{0}(x) \int_{0}^{t} h(s) \frac{f\left((1+K)\left\|e^{-s\mathcal{L}} u_{0}\right\|_{L^\infty(\mathbb{R}^d)}\right)}{\left\|e^{-s\mathcal{L}} u_{0}\right\|_{L^\infty(\mathbb{R}^d)}} d s.
\end{aligned}
$$
Therefore, for $u_{0}=\mu \varphi$ with $\mu>0$, it follows from the definition of the majorant function, inequality \eqref{f}, together with the non-decreasing properties of the functions $\frac{f(y)}{y}$ and $f(y)$ that 
$$
\begin{aligned}
y_{k+1}(x, t) & \leq e^{-t\mathcal{L}} u_{0}(x)+e^{-t\mathcal{L}} u_{0}(x) \int_{0}^{t} h(s) \frac{f\left(\mu(1+K)\left\|e^{-s\mathcal{L}} w_{0}\right\|_{L^\infty(\mathbb{R}^d)}\right)}{\mu\left\|e^{-s\mathcal{L}} w_{0}\right\|_{L^\infty(\mathbb{R}^d)}} d s \\
& \leq e^{-t\mathcal{L}} u_{0}(x)+\frac{f_{M}(\sqrt{\mu}(1+K))}{\sqrt{\mu}} e^{-t\mathcal{L}} u_{0}(x) \int_{0}^{t} h(s) \frac{f\left(\sqrt{\mu} \left\|e^{-s\mathcal{L}} w_{0}\right\|_{L^\infty(\mathbb{R}^d)}\right)}{\sqrt{\mu}\left\|e^{-s\mathcal{L}} w_{0}\right\|_{L^\infty(\mathbb{R}^d)}} d s \\
& \leq e^{-t\mathcal{L}} u_{0}(x)+ \epsilon (1+K)  e^{-t\mathcal{L}} u_{0}(x) \int_{0}^{t} h(s) \frac{f\left(\sqrt{\mu} \left\|e^{-s\mathcal{L}} w_{0}\right\|_{L^\infty(\mathbb{R}^d)}\right)}{\sqrt{\mu}\left\|e^{-s\mathcal{L}} w_{0}\right\|_{L^\infty(\mathbb{R}^d)}} d s \\
& \leq (1+K)e^{-t\mathcal{L}} u_{0}(x),
\end{aligned}
$$
where in the penultimate inequality we use \eqref{3.7} and for the last inequality we choose $\epsilon>0$ such that $(1+K) \epsilon<1.$
This implies that inequality \eqref{3.8} holds for every $j \geq 0$. It is straightforward to show that $y_{j} \leq y_{j+1}$ for all $j \geq 0$ using the non-decreasing property of $f.$ Therefore, using the monotonicity of the sequence $\{y_j\}_{(j \geq 0)}$ with upper bound given by \eqref{3.8}, applying the monotone convergence theorem, we have that the limit $\lim\limits_{j \rightarrow \infty} y_{j}(x, t)$ exists globally. By the definition of the sequence $\{y_j\}_{j\geq 0}$ it follows that
$$
\lim\limits_{j\rightarrow\infty}y_j(x, t)=e^{-t\mathcal{L}} u_{0}(x)+\int_{0}^{t} h(s) e^{-(t-s)\mathcal{L}} f(\lim\limits_{j\rightarrow\infty}y_j(x, s)) d s.$$ On the other hand, relation \eqref{mild} shows that $$\lim\limits_{j\rightarrow\infty}y_j(x, t)=u(x,t),\,(x,t)\in\mathbb{R}^d\times\mathbb{R}_+,$$ and therefore, by passing to the limit as 
$j\rightarrow\infty$ in inequality \eqref{3.8}, we have
\begin{equation*}
0 \leq u(x, t) \leq(1+K) e^{-t\mathcal{L}} u_{0},\,\, (x,t)\in\mathbb{R}^d\times[0,\infty),
\end{equation*} which implies that $u$ exists globally. Thus, a contradiction arises because, according to (ii), no global solution exists. This completes the proof of this theorem.
\end{proof}

\section{Proofs of Corollaries \ref{cor1}-\ref{cor4}} \label{sec3}

In this section, we will present the proofs of Corollaries \ref{cor1}-\ref{cor4}. 
\begin{proof}[Proof of Corollary \ref{cor1}] To conclude this result, we will use the equivalence of (ii) and (iii) of Theorem \ref{th1} as the hypothesis on $f$ and $h$ of Theorem \ref{th1} satisfied by this choice of $f$ and $h$ as $f_m=f=f_M=u^p$ in this case.  The statement of the corollary follows by noting that 
 \begin{align*} \int_{1}^{\infty}  t^{\frac{d}{2\sigma}} \epsilon^{-\frac{dp}{2\sigma}} t^{-\frac{dp}{2\sigma}} d t&= \epsilon^{-\frac{dp}{2\sigma}} \int_{1}^{\infty}    t^{-\frac{d}{2\sigma}(p-1)} d t =\infty\end{align*}
 for every $\epsilon>0$ if and only $-\frac{d}{2\sigma}(p-1)+1 \geq 0$ which is equivalent to $1<p \leq 1+\frac{2\sigma}{d}.$
\end{proof}

\begin{proof}[Proof of Corollary \ref{cor2}] In this case, we $h(t)=t^r,$ $(r \geq 0)$ and $f=(1+u)[\ln(1+u)]^p, (p >1).$ Thus, $0 \leq h \in C([0, \infty))$ and $f$ is a nonnegative convex continuous  function on $[0, \infty)$ with $f(0)=0$ and $f(s)>0$ for $s>0.$ By the definition of minorant and majorant functions, we calculate that 
$$ f_m(u):= \min \left\{ u^p,\, \frac{(1+u)(\log(1+u))^p}{2 (\log 2)^p} \right\},$$ and $$\quad f_M(u):=\max \left\{ u^p,\, \frac{(1+u)(\log(1+u))^p}{2 (\log 2)^p} \right\}.$$
From this, we note that 
    $$ \int_1^\infty \frac{du}{f_m(u)}<+\infty \quad \text{and} \quad \lim_{u \rightarrow 0^+} \frac{f_M(u)}{u}=0.  $$
    Therefore, the hypothesis on Theorem \ref{th1} is verified with this choice of $h$ and $f.$ 
First, we note that for $\log(1+x) \asymp x$ for $x \in [0, 1].$ Now, by using sufficient large $t,$ say, $t \geq T_0$ such that $ \epsilon t^{-\frac{d}{2\sigma}} \in [0, 1]$ and so that $\log(1+ \epsilon t^{-\frac{d}{2\sigma}} ) \asymp \epsilon t^{-\frac{d}{2\sigma}}$ for $\epsilon>0.$ Now, the statement of Corollary follows from Theorem \ref{th1} as 
\begin{align*}
    \int_1^\infty t^r  t^{\frac{d}{2\sigma}} (1+ \epsilon t^{-\frac{d}{2\sigma}} ) [\log(1+ \epsilon t^{-\frac{d}{2\sigma}} )]^p dt  & \asymp \int_1^{T_0} t^r  t^{\frac{d}{2\sigma}} (1+ \epsilon t^{-\frac{d}{2\sigma}} ) [\log(1+ \epsilon t^{-\frac{d}{2\sigma}} )]^p dt\\&+ \int_{T_0}^\infty t^r  t^{\frac{d}{2\sigma}} (1+ \epsilon t^{-\frac{d}{2\sigma}} ) \epsilon^p t^{-\frac{dp}{2\sigma}} dt \\& \asymp   \int_1^{T_0} t^r  t^{\frac{d}{2\sigma}} (1+ \epsilon t^{-\frac{d}{2\sigma}} ) [\log(1+ \epsilon t^{-\frac{d}{2\sigma}} )]^p dt\\&+ \int_{T_0}^\infty \epsilon^{p+1}   t^{r-\frac{dp}{2\sigma}}   dt +\int_{T_0}^\infty \epsilon^{p}   t^{r-\frac{dp}{2\sigma}+\frac{d}{2\sigma}}   dt \\&\asymp \infty 
\end{align*} for every $ \epsilon>0$ if and only if $1+r-\frac{dp}{2\sigma} \geq 0,$ that is, $1<p\leq  \frac{2\sigma(1+r)}{d}.$ \end{proof}

\begin{proof}[Proof of Corollary \ref{cor3}]
    By the definition of minorant and majorant functions, it is easy (see \cite{Chung1}) to note that 
    $$ f_m(u)= \min \left\{ \frac{u^p+ u^q}{2}, u^q \right\}\quad \text{and} \quad f_M(u)=\max \left\{ \frac{u^p+ u^q}{2}, u^q \right\}.$$
    From this, we note that 
    $$ \int_1^\infty \frac{du}{f_m(u)}<+\infty \quad \text{and} \quad \lim_{u \rightarrow 0^+} \frac{f_M(u)}{u}=0.  $$
    Therefore, the hypothesis on Theorem \ref{th1} is verified with this choice of $h$ and $f.$ Therefore, we again use the equivalence of (ii) and (iii) of Theorem \ref{th1}. We note that 
    \begin{align*}
        \int_1^\infty & (t^r+t^s) t^{\frac{d}{2\sigma}} (\epsilon^p t^{-\frac{pd}{2 \sigma}}+\epsilon^q t^{-\frac{qd}{2 \sigma}}) dt \nonumber \\ & = \int_1^\infty \left[ \epsilon^p(t^{r-\frac{d}{2\sigma}(p-1)}+t^{s-\frac{d}{2\sigma}(p-1)})+\epsilon^q(t^{r-\frac{d}{2\sigma}(q-1)}+t^{s-\frac{d}{2\sigma}(q-1)}) \right] dt\\& = +\infty
    \end{align*}
    for every $\epsilon>0$ if and only if $s-\frac{d}{2\sigma}(p-1)+1 \geq 0$ which is equivalent to $1<p\leq  1+\frac{2 \sigma(1+s)}{d}.$
\end{proof}
\begin{proof}[Proof of Corollary \ref{cor4}] To prove this corollary, we will use the equivalence of (ii) and (iii) of Theorem \ref{th1} as the hypothesis on $f$ and $h$ of Theorem \ref{th1} satisfied by this choice of $f=u^p$ and $h=e^{\theta t},\, \theta \neq 0.$  The statement of the corollary follows by noting that 
 $$ \int_{1}^{\infty} e^{\theta t} t^{\frac{d}{2\sigma}} \epsilon^{-\frac{dp}{2\sigma}} t^{-\frac{dp}{2\sigma}} d t= \epsilon^{-\frac{dp}{2\sigma}} \int_{1}^{\infty}  e^{\theta t}  t^{-\frac{d}{2\sigma}(p-1)} d t =\infty$$
 for every $\epsilon>0$ if and only $\theta >0.$
\end{proof}

\section*{Acknowledgments} VK is supported by the FWO Odysseus 1 grant G.0H94.18N: Analysis and Partial Differential Equations, the Methusalem programme of the Ghent University Special Research Fund (BOF) (Grant number 01M01021) and by FWO Senior Research Grant G011522N. BT is supported by the Science Committee of the Ministry of Education and Science of the Republic of Kazakhstan (Grant No. AP23483960).

\section*{Declaration of competing interest} The authors declare that there is no conflict of interest.

\section*{Data Availability Statements} The manuscript has no associated data.


\begin{thebibliography}{AKL92}


\bibitem{BB98} C. Bandle and H. Brunner. Blowup in diffusion equations: a survey. {\it J. Comput. Appl. Math.} 97(1-2) (1998), 3-22. 

\bibitem{BM25} R. B. Belgacem and M. Majdoub. Blow-up for a Nonlocal Diffusion Equation with Time Regularly Varying Nonlinearity and Forcing. ArXiv, (2025).  arXiv:2509.07405v1.

\bibitem{Biagi} S. Biagi, F. Punzo, and E. Vecchi. Global solutions to semilinear parabolic equations driven by mixed local-nonlocal operators. {\it Bull. London Math. Soc.}, 57 (2025), 265--284. 

\bibitem{Castillo3} B. Carhuas-Torre, R. Castillo, R. Freire, A. Lira, M. Loayza, Global and nonglobal solutions for a mixed local-nonlocal heat equation. ArXiv, (2025), arXiv:2505.20401.

\bibitem{Castillo1} R. Castillo, M. Loayza, On the critical exponent for some semilinear reaction–diffusion systems on general domains. {\it J. Math. Anal. Appl.}, 428 (2015) 1117--1134.

\bibitem{Castillo2} R. Castillo, M. Loayza, C.S. Paixão, Global and nonglobal existence for a strongly coupled parabolic system on a general domain. {\it J. Differ. Equ.}, 261 (2016) 3344--3365.

\bibitem{CKR} M. Chatzakou,  A. Kassymov, and M. Ruzhansky, On global solutions of heat equations with time-dependent nonlinearities on unimodular Lie groups. (2024). arXiv preprint arXiv:2404.05611 

\bibitem{Chen1} Z.-Q. Chen, P. Kim, R. Song, Global heat kernel estimates for $\Delta+\Delta^{\alpha/2}$ in half-space-like domains, {\it Electron. J. Probab.}, 17 (32) (2012) 32.

\bibitem{Chen2} Z.-Q. Chen, E. Hu, Heat kernel estimates for $\Delta+\Delta^{\alpha/2}$ under gradient perturbation, {\it Stoch. Process. Appl.}, 125 (2015) 2603--2642.

\bibitem{Chung1} S.-Y. Chung, J. Hwang, A necessary and sufficient condition for the existence of global solutions to discrete semilinear parabolic equations on networks. {\it Chaos Solitons Fractals}, 158 (2022), 112055.

\bibitem{Chung} S.-Y. Chung, J. Hwang, A necessary and sufficient conditions for the global existence of solutions to fractional reaction-diffusion equations on $\mathbb{R}^N$. {\it Fract. Calc. Appl. Anal.}, 27:5 (2024), 2606--2619.

\bibitem{Pezzo} L. Del Pezzo and R. Ferreira. Fujita exponent and blow-up rate for a mixed local and nonlocal heat equation. {\it Nonlinear Analysis} 255 (2025): 113761.

\bibitem{Dip1} S. Dipierro and E. Valdinoci. Description of an ecological niche for a mixed local/nonlocal dispersal: An evolution equation and a new Neumann condition arising from the superposition of Brownian and Lévy processes. {\it Physica A: Stat. Mech. Appl.}, 575 (2021), 126052.

\bibitem{Dip2} S. Dipierro, E. P. Lippi, and E. Valdinoci, (Non)local logistic equations with Neumann conditions, {\it Ann. Inst. H. Poincaré C Anal. Non Linéaire}, 40 (2023), 1093--1166.

\bibitem{Dip3} S. Dipierro, E. P. Lippi, and E. Valdinoci, The role of Allee effects for Gaussian and Lévy dispersals in an environmental niche, {\it J. Math. Biol.}, 89 (2024), 19.

\bibitem{Dip4} S. Dipierro, E. P. Lippi, and E. Valdinoci, Some maximum principles for parabolic mixed local/nonlocal operators, {\it Proc. Amer. Math. Soc.}, 152 (2024), 3923--3939.

\bibitem{Fino} A. Z. Fino, M. Kirane, Qualitative properties of solutions to a time-space fractional evolution equation, {\it Quart. Appl. Math.}, 70 (2012) 133--157.

\bibitem{Fujita} H. Fujita. On the blowing up of solutions of the Cauchy problem for $u_t=\triangle u+u^{1+\alpha}$. {\it J. Fac. Sci. Univ. Tokyo Sect.} 13 (1966), 109--124.

\bibitem{Guedda} M. Guedda, M. Kirane, Criticality for some evolution equations, {\it Differ. Equ.}, 37 (2001), 540--550.

\bibitem{Hayakawa} K. Hayakawa, On nonexistence of global solutions of some semilinear parabolic differential equations, {\it Proc. Japan Acad.}, 49 (1973) 503--505.

\bibitem{Kirane} M. Kirane, A. Z. Fino, and A. Ayoub. Decay of mass for a semilinear heat equation with mixed local and nonlocal operators, {\it Fract. Calc. Appl. Anal.}, 28 (2025)
1756–1776.

\bibitem{KT2025} V. Kumar and B. T. Torebek, Fujita-type results for the semilinear heat equations driven by mixed local-nonlocal operators. ArXiv, (2025). arXiv.2502.21273 

\bibitem{Laister} R. Laister, M. Sier\.z\k{e}ga, A blow-up dichotomy for semilinear fractional heat equations. {\it Math. Ann.}, 381(1-2) (2021), 75--90.

\bibitem{LP} M. Loayza and C.S.D.  Paixão, Existence and non-existence of global solutions for a semilinear heat equation on a general domain. {\it Electron. J. Differ. Equ.}, 168 (2014) 9 pp.

\bibitem{Meier} P. Meier, On the critical exponent for reaction-diffusion equations. {\it Arch. Rational Mech. Anal.}, 109:1 (1990), 63--71.

\bibitem{Naga} M. Nagasawa, T. Sirao, Probabilistic treatment of the blowing up of solutions for a non-linear integral equation, {\it Trans. Am. Math. Soc.}, 139 (1969) 301--310.

\bibitem{pazy}
A.~Pazy, {Semigroups of linear operators and applications to partial differential equations}, Vol.~44 of Applied Mathematical Sciences, Springer-Verlag, New York, 1983.

\bibitem{QS} P. Quittner and P. Souplet; Superlinear parabolic problems, Blow-up, global existence and steady sates, Birkhauser Verlag AG, 2007

\bibitem{Sugitani} S. Sugitani, On nonexistence of global solutions for some nonlinear integral equations, {\it Osaka Math. J.}, 12 (1975) 45--51.

\end{thebibliography}
\end{document}